\documentclass[12pt,dvips]{article}

\usepackage{amsmath,amsfonts,amssymb,amsthm,graphicx,verbatim,subfig}
\usepackage[all]{xy}

\ifx\pdftexversion\undefined

\usepackage[a4paper,colorlinks,linkcolor=black,filecolor=black,citecolor=black,urlcolor=black,pdfstartview=FitH]{hyperref}
\fi

\font\sixbb=msbm6
\font\eightbb=msbm8
\font\twelvebb=msbm10 scaled 1095
\newfam\bbfam
\textfont\bbfam=\twelvebb \scriptfont\bbfam=\eightbb
                           \scriptscriptfont\bbfam=\sixbb
\def\bb{\fam\bbfam\twelvebb}
\newcommand{\Rea}{{\bb R}}

\newtheorem{theorem}{\bf Theorem}[section]
\newtheorem{claim}[theorem]{\bf Claim}

\newtheorem{lemma}{Lemma}

\newcommand{\beq}[0]{\begin{equation}}
\newcommand{\enq}[0]{\end{equation}}

\newcommand{\prob}{{\rm Pr}}

\newcommand{\namedref}[2]{\hyperref[#2]{#1~\ref*{#2}}}

\begin{document}
\title{The threshold for collapsibility in random complexes}
\author{L. Aronshtam\thanks{Department of Computer Science, Hebrew University, Jerusalem 91904,
    Israel. e-mail: liobsar@gmail.com~.} \and  N. Linial\thanks{Department of Computer Science, Hebrew University, Jerusalem 91904,
    Israel. e-mail: nati@cs.huji.ac.il~. Supported by ISF and BSF grants.}
}

\maketitle

\begin{abstract}
In this paper we determine the threshold for collapsibility in the probabilistic model $X_d(n,p)$ of $d$-dimensional simplicial complexes. A lower bound for this threshold $p=\frac{c_d}{n}$ was established in \cite{ALLM}. Here we show that this is indeed the correct threshold. Namely, for every $c>c_d$, a complex drawn from $X_d(n,\frac{c}{n})$ is asymptotically almost surely not collapsible.
\end{abstract}
Keywords: Random simplicial complex, threshold, collapsibility, topology.
\section{A soft introduction}\label{intro}
We work with the probability space $X_d(n,p)$ of $n$-vertex $d$-dimensional simplicial complexes introduced in \cite{LM}. Such a complex $X$ has a full $(d-1)$-dimensional skeleton, and each $d$-dimensional face is taken to be in $X$ independently and with probability $p$. A $(d-1)$-dimensional face $\tau$ of $X$ is {\it isolated} if it is not contained in any $d$-face. The $(d-1)$-face $\tau$ is {\it free} if it is contained in exactly one $d$-dimensional face $\sigma$ of $X$. In this case there is a corresponding {\em elementary collapse} step, in which $\tau$ and $\sigma$ are removed from $X$. We say that $X$ is {\em collapsible} if it is possible to eliminate all its $d$-faces by a series of elementary collapses. It is useful to note that a graph (i.e., a $1$-dimensional complex) is collapsible iff it is a forest. A graph from $G(n,p)$ is a forest with probability $o_n(1)$ for $p>\frac 1n$ and the bound is tight. Establishing the threshold for collapsibility in $X_d(n,p)$ can be viewed as the $d$-dimensional analog of this fact. This problem was previously investigated in \cite{CFK10,ALLM}. The present article shows that the bound established in \cite{ALLM} is tight.

We are inspired by work on cores in random graphs and in particular by Molloy's work~\cite{Molloy}. The basic strategy of the proof is this: In~\cite{ALLM} we found a constant $c_d$ and proved that if $c' < c_d$, then asymptotically almost surely (=a.s.) a random complex in $X_d(n, \frac{c'}{n})$ is either collapsible or it contains the boundary of a $(d+1)$-dimensional simplex. Now we want to prove that for $c > c_d$, a random complex in $X_d(n, \frac{c}{n})$ is a.s. not collapsible. In our analysis, we split the collapsing process to two {\em epochs} as follows. Fix a positive integer $r$ (to be specified below) and start by carrying out $r$ {\em phases} of collapses. In each phase we simultaneously collapse on {\em all} $(d-1)$-faces that are presently free. After $r$ such phases we move on to the second epoch of the process in which we collapse free $(d-1)$-faces one by one. We say two $(d-1)$-faces are neighbors if there exists a $d$-face that contains both of them. The order in which we collapse in the second epoch is the following: (i) Mark all the free faces in the complex, if there exists two or more neighboring faces that are free, mark only one of them. (ii) Sample uniformly a permutation $\pi$ on all the marked faces. (iii) Collapse the free $(d-1)$-faces one by one by the order of $\pi$. (iiii) Repeat stages (i)-(iii) until there are no free $(d-1)$-faces left. A face that gets marked during the $k$-th run through stages (i)-(iii) is said to have {\em mark} $k$.

If the integer $r$ is large enough, then the results from~\cite{AL} give us a good idea about several relevant random variables. Let us denote by $\Delta(\tau)$ the degree of the $(d-1)$-face $\tau$ (i.e., the number of $d$-faces that contain $\tau$) at the end of the first epoch. For a given $d$-face $\sigma$ we let $\Delta^{\sigma}(\tau):=\Delta(\tau)-1$ (resp. $\Delta^{\sigma}(\tau):=\Delta(\tau)$) if $\tau\subset\sigma$ (resp. $\tau\not\subset\sigma$). The random variable $X_0$ counts the number of $(d-1)$-faces $\tau$ that are free, i.e., $\Delta(\tau)=1$ at the end of the first epoch. Also, $L=L_r$ is the number of $(d-1)$-faces $\tau$ for which $\Delta(\tau)>0$. We denote by $X_i$ the number of free $(d-1)$-faces at step $i$ of the second epoch. Our plan is to show that the expected drop in this number $\mathbb{E}(X_i-X_{i+1})$ is sufficiently large to a.s. guarantee that at some moment the complex still has some $(d-1)$-faces, but none of them are free. This clearly means that the original complex is non-collapsible.

The main ingredient in our proof is then, a detailed accounting of the free $(d-1)$-faces. In every step this number gets decreased by one due to the loss of the free $(d-1)$-face that we are presently collapsing. What complicates matters is that as we carry out the collapse step we may be creating some new free $(d-1)$-faces. We define the event $S^{j}_i$ that our $i$-th collapsing step creates $j$ new free $(d-1)$-faces. Note that $0\le j\le d$, and in particular $X_i \le X_{i-1}+j-1$. We further introduce the random variable $Y_i$ that counts the number of new free $(d-1)$-faces added in the $i$-th collapsing step.

Clearly
\begin{equation}\label{big_plan}
X_i=X_{i-1}-1+Y_i=X_0-i+\sum_{l=1}^{i}Y_l
\end{equation} 
To recap: When $X_i=0$, we run out of free faces, and if this happens before the complex becomes empty, the original complex is non-collapsible.

We now recall a basic technique from~\cite{AL}. A $d$-tree has the following recursive definition: (i) A single $d$-face is a $d$-tree. (ii) A $d$-tree with $n$ faces is constructed by taking a $d$-tree with $n-1$ faces $T$, choosing a $(d-1)$-face $\tau$ in $T$, and a new vertex $u$, and adding $\tau\cup  u$ and all its subfaces to $T$. A $d$-tree is a rooted $d$-tree in which one $(d-1)$-face is designated to be the root. Two neighboring $(d-1)$-faces have distance $1$ and so we can talk about the distance between every two $(d-1)$-faces in a $d$-tree. We consider the probability space, ${\cal T }(c,t)$, of rooted $d$-trees of radius at most $t$  (i.e every $(d-1)$-face in the tree  is at distance at most $t$ from the root). The probability space ${\cal T } (c,0)$ clearly contains only the $d$-tree that is just the root.  A tree is sampled from ${\cal T }(c,t)$ by (i) Sampling a tree $T$ from ${\cal T }(c,t-1)$ (ii) For every $(d-1)$-face, $\tau$ in $T$, at distance $t-1$ from the root, create $l$ new vertices $v_1,\dots,v_l$ and add the $d$-faces $\{\tau\cup v_1,\dots,\tau\cup v_l\}$ and their subfaces to $T$. The integer $l$ is sampled from the Poisson distribution with parameter $c$. We are interested in $\tau$-collapsing of $d$-trees where we collapse the faces of the $d$-tree {\em but not the root} $\tau$.

As shown in~\cite{AL}, for every $(d-1)$-face $\tau$ in $X\in X_d(n,\frac cn)$ the $t$-th neighborhood of $\tau$ is a.s a $d$-tree. Moreover, the distributions of such $d$-trees is very close to the distribution of ${\cal T }(c,t)$ of $d$-trees rooted at $\tau$. This simplifies the analysis of the $\tau$-collapse process on a local scale, where we run the collapse phases as usual, except that we {\em forbid collapsing} on $\tau$. We denote by $\gamma_t(c)$ (resp. $\beta_t(c)$) the probability that $\tau$ becomes isolated in fewer than (resp. more than) $t$ collapse phases. Obviously $\beta_t(c) = 1-\gamma_{t+1}(c)$. The following relations are proved in~\cite{AL}:
\[\gamma_0(c)=0,~~~\gamma_{t+1}(c)=\exp(-c(1-\gamma_t(c)^d))\mbox{~~for~~}t=0,1,\ldots.\]
To simplify notations, and as long as everything is clear, we suppress the dependence on $c$ and write $\gamma_t,\beta_t$ rather than $\gamma_t(c),\beta_t(c)$.

We now use this model and consider a normal run of the first epoch with the change that we forbid collapsing on $\tau$ (i.e a $\tau$-collapse). Denote by $A_k^{\tau}$ the event (in this modified process) that $\tau$ has degree $k$ at the end of the first epoch. For $k \ge 0$
\begin{eqnarray}\label{A_k}
\Pr(A_k^{\tau})=\sum_{j=k}^{\infty}\frac{c^j}{j!}e^{-c}{j\choose k} (1-(1-\gamma_{r})^d)^{j-k}(1-\gamma_{r})^{kd}=\nonumber
\frac{((\beta_{r-1})^{d}c)^k}{k!}\gamma_{r+1}~\nonumber.
\end{eqnarray}

Consequently, the degree of ${\tau}$ is Poisson distributed with parameter $(\beta_{r-1})^{d}c$.

\section{In detail}
As mentioned, the first epoch proceeds for $r$ phases. At this point the number of $(d-1)$-faces of positive degree in the complex is $L_{r}$. As shown in \cite{AL}:
$$
\mathbb{E}[L_r]=(1-o(1)){n \choose d} (1-(\gamma_{r+1}+c\gamma_r\beta^d_{r-1})).
$$
A concentration of measure argument similar to the one used in \cite{AL} yields that a.s $|\mathbb{E}[L_r]-L_r|\le o(n^d).$

For a $(d-1)$-face to be free at the end of the first epoch, it must become free only in the last collapsing phase. Therefore the probability for a $(d-1)$-face to be free in the beginning of the second epoch is:
\begin{eqnarray}\label{X_0}
& &\sum_{j=1}^{\infty}\frac{c^j}{j!}e^{-c}j (1-(1-\gamma_{r})^d)^{j-1}(1-\gamma_{r})^{d}-\sum_{j=1}^{\infty}\frac{c^j}{j!}e^{-c}j (1-(1-\gamma_{r-1})^d)^{j-1}(1-\gamma_{r})^{d}=\nonumber\\
& &  (\beta_{r-1})^{d}c\gamma_{r+1}-(\beta_{r-1})^{d}c\gamma_{r}=(\beta_{r-1})^{d}c(\gamma_{r+1}-\gamma_{r})~\nonumber.
\end{eqnarray}
Consequently, the expectation of $X_0$ is:
\begin{eqnarray}\label{expx_0}
& & \mathbb{E}[X_0]={n\choose d}(\beta_{r-1})^{d}c(\gamma_{r+1}-\gamma_{r}).
\end{eqnarray}

Let $\sigma\supset\tau$ be faces that we collapse at some step of the second epoch. The $(d-1)$-faces {\it affected} by that step are all the $(d-1)$-subfaces of $\sigma$ other than $\tau$. As mentioned, we control the number of free $(d-1)$-faces throughout the process. Clearly the change in this number at a given step is determined by the current degrees of the affected faces. We consider the count of steps in the second epoch as ``time". In particular, time zero means the end of the first epoch and the beginning of the second one.

Let $D_k$ be the random variable that counts the $(d-1)$-faces $\tau$ that have degree $k$ at time zero of the $\tau$-process. Clearly
\begin{eqnarray}\label{D_k}
& & \mathbb{E}[D_k]=(1-o(1))\sum_{\tau}\Pr(A_k^{\tau})=(1-o(1)){n\choose d}\frac{((\beta_{r-1})^{d}c)^k}{k!}\gamma_{r+1}.
\end{eqnarray}
Let $B_j$ be the random variable that counts $(d-1)$-faces which are not isolated after the $j$-th collapsing phase of the first epoch. Obviously  $\mathbb{E}[B_j]=(1-o(1)){n\choose d}\beta_{j}$. Again a concentration of measure argument in the spirit of \cite{AL} yields that a.s $|\mathbb{E}[D_k]-D_k|\le o(n^d)$ and $|\mathbb{E}[B_j]-B_j|\le o(n^d)$.

Let $\sigma\supset\tau$ be the faces that we collapse at time $i$. We are interested in the event that some affected $(d-1)$-face $\tau'$ becomes free due to this step. This can happen if $\Delta^{\sigma}(\tau')=k>1$ and each of the other $k-1$ $d$-faces that contains $\tau'$ gets collapsed in time $<i$. The other way in which this can happen is the event $C^{\sigma,\tau'}$ that $\Delta^{\sigma}(\tau')=1$.  

We next define the random variable $D'_k$ exactly as $D_k$, except that ``degree" is interpreted a little differently. Namely a $(d-1)$-face of degree $s$ that is contained in $\sigma$ is considered as having degree $s-1$. Clearly $D_k\ge D'_k\ge D_k-(d+1)$ and the same concentration of measure arguments work for both random variables. Since $\sigma$ was not collapsed in the first epoch, we know that $\tau'$ was not collapsed, hence the degree of $\tau'$ in $X\setminus\sigma$ after the first epoch is the degree in $X\setminus\sigma$ if we forbid collapsing on $\tau'$. 

Let us consider the $d$-tree that consists of the $r$-neighborhood of $\tau'$ barring the $d$-face $\sigma$. Before the last phase of the first epoch, $\tau'$ does not become isolated in this tree. Otherwise $\tau'$ would have become a free face of $\sigma$ so that $\sigma$ would have been collapsed in the first epoch. Therefore  
\begin{eqnarray}\label{C}
\Pr(C^{\sigma,\tau'})=\frac{\mathbb{E}[D'_1]}{\mathbb{E}[B_{r-1}]}=(\beta_{r-1})^{d-1}c\gamma_{r+1}.
\end{eqnarray}

We denote this probability by $x:=(\beta_{r-1})^{d-1}c\gamma_{r+1}$.

Let ${\cal{Q}}^{\tau'}_i$ be the event that $i$ is the first time in which $\tau'$ is an affected face. The number of $(d-1)$-faces that get affected before time $i$ cannot exceed $(i-1)d$. Therefore,
\begin{eqnarray}\label{Q^{t}_i}
\Pr({\cal{Q}}^{\tau'}_i)\ge 1-\frac{(i-1)d}{\beta_{r-1}{n\choose d}}
\end{eqnarray}

We define ${\cal{Q}}_i$ as the event that for each face that is affected in step $i$ this is the first time that it gets affected. We prove

\begin{lemma}\label{Q_i}
For every $i$ there holds
$$\Pr({\cal{Q}}_i)\ge \prod_{\tau_j\in T}\Pr({\cal{Q}}^{\tau_j}_i).$$
Here $T=\{\tau_1\dots,\tau_d\}$ is the set of faces that are affected in collapsing step $i$. 
\end{lemma}
\begin{proof}
Let $\tau$ be the $(d-1)$-face collapsed in the $i$-th collapsing step, and let $k$ be its mark. If some face $\tau_j\in T$ has a neighbor with mark $<k$ then $\Pr({\cal{Q}}^{\tau_j}_i)=0$, $\Pr({\cal{Q}}_i)=0$ and the inequality clearly holds. Otherwise, every face in $T$ is affected in a collapsing step where the collapsed $(d-1)$-face has mark $k$. Hence the probability of the event ${\cal{Q}}^{\tau_j}_{i}$ depends only on the random order selected for marking step $k$.  Specifically $\Pr{\cal{Q}}^{\tau_j}_i=\frac{1}{\alpha_j}$, where $\alpha_j$ be the number of neighbors of $\tau_j$ with mark $k$.
$$\Pr({\cal{Q}}_i)\ge \frac{1}{\sum_J\alpha_j}$$
where $J$ is the set of indices $j$ for which $\alpha_j > 1$. But for all $j$,
$$\prod_j\Pr({\cal{Q}}^{\tau_j}_i)=\prod_J\Pr({\cal{Q}}^{\tau_j}_i)=\prod_J\frac{1}{\alpha_j}$$
and the conclusion follows, since $\frac{1}{\sum_J\alpha_j} \ge \prod_J\frac{1}{\alpha_j}$.
\end{proof}

In view of~(\ref{Q^{t}_i}) and Lemma~\ref{Q_i}, we see that if $i\ll \beta_{r-1}{n\choose d}/d$ then $\Pr({\cal{Q}}^{\tau'}_i)= 1-o(1)$ and hence $\Pr({\cal{Q}}_i)=1-o(1)$.

We can finally calculate $\mathbb E(Y_i)$, the expected number of new free $(d-1)$ faces added in a collapsing step. We do this in terms of the events $S^j_i$.




It is convenient for us to condition on the almost sure event ${\cal{Q}}_i$.

 For every $1\le j\le d-1$ we write
\begin{eqnarray}\label{S^j_i}
\nonumber & &\Pr(S^j_i)=\Pr(S^j_i|{\cal{Q}}_i)\Pr({\cal{Q}}_i)+\Pr(S^j_i|\bar{{\cal{Q}}_i})\Pr(\bar{{\cal{Q}}_i})\le \\\nonumber
& &\le\Pr(S^j_i|{\cal{Q}}_i)+1-\Pr({{\cal{Q}}_i})\le\\\nonumber
& &\le (1+o(1)) {d\choose j}(1-x)^{d-j}{x}^{j}
\end{eqnarray}
To see this, notice that conditioned on ${\cal{Q}}_i$ the event $S^j_i$ happens if $\Delta(\tau')=1$ for exactly $j$ of the affected faces $\tau'$.

Notice $\Pr(Y_i=j)=\Pr(S_i^j)$, hence
\begin{equation}\label{shesh}
\mathbb{E}[Y_i]\le(1+o_n(1)) dx=(1+o_n(1)) d(\beta_{r-1})^{d-1}c\gamma_{r+1}
\end{equation}

\begin{claim}\label{Y_i<1}
For $i\ll \beta_{r-1}{n\choose d}/d$ there holds $\mathbb{E}[Y_i]<1.$ 
\end{claim}
\begin{proof}
Recall that $d\cdot\beta^{d-1}c_d\gamma=1$ where $\beta=\lim_{r\to\infty} \beta_r(c_d))$ and $\gamma=1-\beta$. We turn to show that $d\cdot\beta_{r-1}^{d-1}c\gamma_{r+1}<1$ for $c>c_d$. To this end we recall and slightly extend some analysis from~\cite{ALLM, AL}. The role of the parameter $c_d$ is revealed by analyzing the function $f_c(t)=1-e^{-c t^d}-t$. As it turns out, the behavior of this function changes significantly as the parameter $c$ varies around $c_d$ which is the range in which we work. 

Concretely we note that for $c>c_d$ the function $f_c$ has two roots in the range $1>t>0$. We denote these roots by $0< b(c)< B(c)< 1$. Note that $B(c)=\lim_{k\to\infty} \beta_k(c)$. The fact that $b(c), B(c)$ are roots of $f_c$ yields $c=g(B(c))=g(b(c))$ where $g(t)=\frac{-\ln(1-t)}{t^d}$. Therefore $g'(b(c))b'(c)=g'(B(c))B'(c)=1$. (Prime denotes derivative w.r.t $c$). We claim that $B$ (resp. $b$) is an increasing (resp. decreasing) function of $c$, which would follow if $g'(b(c))<0<g'(B(c))$ . This is indeed so since, as is easily verified, $g$ has one extremal point in $(0,1)$ which is a minimum, and $0<b<B<1$.

Let $h(c)= d\cdot B(c)^{d-1}\cdot c\cdot(1-B(c))$, and recall that $h(c_d)=1$. As we show next, $h(c)<1$ for $c$ that is a little larger than $c_d$. We denote derivatives w.r.t $B$ by an upper dot, and we prove this claim by showing that $\dot{h}<0$. By choosing $r$ large enough we obtain $d(\beta_{r-1})^{d-1}c\gamma_{r+1}<1$. Together with~(\ref{shesh}) this implies the claim.

We take the derivative w.r.t $B$ of the relation $c=g(B)$:
\[
\dot{c}=\frac{1}{B(c)^d(1-B(c))}+\frac{d\cdot\ln(1-B(c))}{B(c)^{d+1}} = \frac{1}{B(c)^d(1-B(c))}-\frac{d\cdot c}{B(c)}.
\]

But
\[
\dot{h}=d\left(((d-1)B(c)^{d-2}-dB(c)^{d-1})c+(B(c)^{d-1}-B(c)^d)\dot{c}\right)=
\]
\[=-c\cdot B(c)^{d-2}+\frac{1}{B(c)}=
\frac{\ln(1-B(c))}{B(c)^2}+\frac{1}{B(c)}
\]
which is negative, since $-\ln(1-u)>u$ for all $u$.

\end{proof}
To complete our proof we need to show that for some $i\le \eta{n\choose d}$ small enough, a.s $X_i=0$.
As was done in~\cite{Molloy}, we define a series of random variables $Z_0,Z_1,\dots$ (that are very similar to $X_0,X_1,\dots$) as follows. We let $Z_0=X_0$. For every $i>0$, $Z_i:=Z_{i-1}-1+Y'_i$, where $Y'_i$ has the same distribution as $Y_i$. In other words, for every $0\le j\le d$ the random variable $Z_i$ takes the value $Z_{i-1}-1+j $ with probability $\Pr(S^j_i)$. The simple but crucial fact is that if $Z_i<0$ then $X_i=0$. By choosing $r$, the number of phases in the first epoch, large enough we can make $\mathbb{E}[X_0]=\delta {n\choose d}$ with an arbitrarily small $\delta >0$. As shown in Claim \ref{Y_i<1}, there exists an $\epsilon$ such that $\mathbb{E}[Y_i]=\mathbb{E}[Y'_i]\le 1- \epsilon$. Thus for $i>\frac{2\delta}{\epsilon} {n\choose d}$ $$\mathbb{E}[Z_i]=\mathbb{E}[X_0]-i+dxi<-\delta {n\choose d}.$$ The desired conclusion follows now from a version of Azuma's inequality from \cite{McDiarmid}: 
\begin{theorem}
\label{mcazuma}
Let $X_1\dots,X_n$ be independent random variables taking values in a set $A$. Suppose that the function  $f:A^n \rightarrow \Rea$ has the property that
$$
|f(t)-f(t')| \leq \epsilon_k,
$$
if $t,t'\in A^n$ differ only in their $k$-th coordinates. Then, for every $l>0$:
\begin{equation}
\label{azumeq}
\prob[|f(X_1,\dots ,X_n)-\mathbb{E}[f(X_1,\dots ,X_n)]|\ge l]\le 2e^{-\frac{2\cdot l^2}{\sum_k \epsilon_k^2}}
\end{equation}
\end{theorem}
Fix an $i$ that is bigger than $\frac{2\delta}{\epsilon} {n\choose d}$ and smaller than $\beta_{r-1}{n\choose d}/d$ and $L_r$. The random variables $Y_1',\dots,Y_i'$ take values in $\{0,\dots,d\}$. Notice $Z_i=f(Y_1',\dots,Y_i')=X_0-i+\sum_{j=0}^i Y_j'$ and for two vectors $t,t'$ that differ in only the $k$-th coordinate $|f(t)-f(t')| \le d$, hence for $l=\delta{n\choose d}$
$$
\Pr[|f(Y_1',\dots ,Y_i')-\mathbb{E}[f(Y_1',\dots ,Y_i')]|\ge l]\le 2e^{-\frac{2\cdot l^2}{i\cdot d^2}}=o(1).
$$
Thus a.s $Z_i<0$, as claimed.


\begin{thebibliography}{99}

\bibitem{AL} 
L. Aronshtam and N. Linial, When does the top homology of a random simplicial complex vanish?, {\it Random Struct. Algorithms}, (2013) doi: 10.1002/rsa.20495

\bibitem{ALLM}
L. Aronshtam, N. Linial, T. {\L}uczak and R. Meshulam, Collapsibility and vanishing of top homology in random simplicial complexes, {\it Discrete and Computational Geometry}, 2013(49)  317-334.

\bibitem{CFK10}
A. Costa, M. Farber and T. Kappeler,
Topology of random 2-complexes, {\it Discrete and Computational Geometry}, 2012(47) 117-149.

\bibitem{LM}
N. Linial and R. Meshulam, Homological connectivity of random 2-complexes, {\it Combinatorica}, 2006(26) 475-487.

\bibitem{McDiarmid}
C. McDiarmid, On the Method of Bounded Differences,
 {\it Surveys in combinatorics 1989}, 148--188,
London Math. Soc. Lecture Note Ser., 141, Cambridge Univ. Press, Cambridge, 1989.

\bibitem{Molloy}
M. Molloy, Cores in random hypergraphs and Boolean formulas, {\it Random Struct. Algorithms}, 27(2005) 124-135


\end{thebibliography}
\end{document}